\documentclass[a4paper,11pt]{article}

\usepackage[utf8]{inputenc}
\usepackage[english]{babel}
\usepackage{amsmath,amssymb,amsthm,,amsbsy,bbm} 
\usepackage{hyperref,xcolor,fullpage}
\usepackage{dsfont}
\usepackage{faktor}
\usepackage{todonotes}
\usepackage{mathtools}
\usepackage{enumitem}

\usepackage{float}

\usepackage{tikz}
\usetikzlibrary{patterns,snakes}
\usetikzlibrary{calc}

\makeatletter \let\@fnsymbol\@arabic \makeatother 

\allowdisplaybreaks 
\usepackage[numbers]{natbib} 

\newtheorem{theorem}{Theorem}[section]
\newtheorem{lemma}[theorem]{Lemma}

\newtheorem{corollary}[theorem]{Corollary}
\newtheorem{question}[theorem]{Question}

\newtheorem{claim}[theorem]{Claim}
\newtheorem{proposition}[theorem]{Proposition}

\theoremstyle{remark}
\newtheorem*{example}{Example}
\newtheorem{remark}[theorem]{Remark}

\makeatletter
\let\save@mathaccent\mathaccent
\newcommand*\if@single[3]{
	\setbox0\hbox{${\mathaccent"0362{#1}}^H$}%
	\setbox2\hbox{${\mathaccent"0362{\kern0pt#1}}^H$}%
	\ifdim\ht0=\ht2 #3\else #2\fi }
\newcommand*\rel@kern[1]{\kern#1\dimexpr\macc@kerna}
\newcommand*\widebar[1]{\@ifnextchar^{{\wide@bar{#1}{0}}}{\wide@bar{#1}{1}}}
\newcommand*\wide@bar[2]{\if@single{#1}{\wide@bar@{#1}{#2}{1}}{\wide@bar@{#1}{#2}{2}}}
\newcommand*\wide@bar@[3]{
	\begingroup
	\def\mathaccent##1##2{
		\let\mathaccent\save@mathaccent
		\if#32 \let\macc@nucleus\first@char \fi
		\setbox\z@\hbox{$\macc@style{\macc@nucleus}_{}$}
		\setbox\tw@\hbox{$\macc@style{\macc@nucleus}{}_{}$}
		\dimen@\wd\tw@ \advance\dimen@-\wd\z@ \divide\dimen@ 3 \@tempdima\wd\tw@ \advance\@tempdima-\scriptspace \divide\@tempdima 10 \advance\dimen@-\@tempdima \ifdim\dimen@>\z@ \dimen@0pt \fi \rel@kern{0.6}\kern-\dimen@
		\if#31 \overline{\rel@kern{-0.6}\kern\dimen@\macc@nucleus\rel@kern{0.4}\kern\dimen@} \advance\dimen@0.4\dimexpr\macc@kerna \let\final@kern#2 \ifdim\dimen@<\z@ \let\final@kern1 \fi
		\if \final@kern1 \kern-\dimen@ \fi
		\else \overline{\rel@kern{-0.6}\kern\dimen@#1} \fi }
	\macc@depth\@ne	\let\math@bgroup\@empty \let\math@egroup\macc@set@skewchar 	\mathsurround\z@ \frozen@everymath{\mathgroup\macc@group\relax} 	 \macc@set@skewchar\relax \let\mathaccentV\macc@nested@a	\if#31 \macc@nested@a\relax111{#1} \else \def\gobble@till@marker##1\endmarker{} \futurelet\first@char\gobble@till@marker#1\endmarker \ifcat\noexpand\first@char A\else \def\first@char{} \fi \macc@nested@a\relax111{\first@char} \fi
	\endgroup }
\makeatother

\newcommand{\miniminus}{\text{-}}
\newcommand{\miniplus}{\text{+}}

\newcommand*\bell{\ensuremath{\boldsymbol\ell}}

\newcommand{\be}{\mathbf{e}}
\newcommand{\bb}{\mathbf{b}}
\newcommand{\bp}{\mathbf{p}}
\newcommand{\bc}{\mathbf{c}}
\newcommand{\bx}{\mathbf{x}}
\newcommand{\by}{\mathbf{y}}
\newcommand{\bz}{\mathbf{z}}
\newcommand{\bu}{\mathbf{u}}
\newcommand{\bv}{\mathbf{v}}
\newcommand{\bw}{\mathbf{w}}

\newcommand{\bH}{\mathbf{H}}
\newcommand{\bP}{\mathbf{P}}
\newcommand{\bC}{\mathbf{C}}

\usepackage{mathrsfs}
\DeclareMathAlphabet{\mathpzc}{OT1}{pzc}{m}{it}
\usepackage{xifthen}
\newcommand{\fp}[1][]{ \ifthenelse{\isempty{#1}}{\mathpzc{p}}{\mathpzc{p}(#1)} }

\newcommand{\mP}{\mathcal{P}}

\newcommand{\mI}{\mathcal{I}}

\newcommand{\NN}{\mathbb{N}}
\newcommand{\ZZ}{\mathbb{Z}}

\newcommand{\ceil}[1]{\lceil {#1} \rceil}

\definecolor{darkred}{cmyk}{.2,.9,.8,0}

\title{Another Note on Intervals in the Hales--Jewett Theorem}
\author{
	Nina Kam\v{c}ev\thanks{School of Mathematical Sciences, Monash University, VIC 3800 Australia. Email: {\tt nina.kamcev@monash.edu}. Visit to Universitat Polit\`ecnica de Catalunya supported by the grant TM2017-82166-P9.} \and
	Christoph Spiegel\thanks{Universitat Polit\`ecnica de Catalunya and Barcelona Graduate School of Mathematics, Department of Mathematics, Edificio Omega, 08034 Barcelona, Spain. E-mail: {\tt christoph.spiegel@upc.edu}. Supported by the Spanish Ministerio de Econom\'{i}a y Competitividad FPI grant under the project MTM2014-54745-P and the Mar\'ia de Maetzu research grant MDM-2014-0445.}
}

\begin{document}
\maketitle

\begin{abstract} 
	The Hales--Jewett Theorem states that any $r$--colouring of $[m]^n$ contains a monochromatic combinatorial line if $n$ is large enough. Shelah's proof of the theorem implies that for $m = 3$ there always exists a monochromatic combinatorial lines whose set of active coordinates is the union of \emph{at most} $r$ intervals. Conlon and Kam\v{c}ev proved the existence of colourings for which it cannot be \emph{fewer} than $r$ intervals if $r$ is odd. For $r = 2$ however, Leader and R\"aty showed that one can always find a monochromatic combinatorial line whose active coordinate set is a single interval. In this paper, we extend the result of Leader and R\"aty to the case of all even $r$, showing that one can always find a monochromatic combinatorial line in $[3]^n$ whose set of active coordinate is the union of at most $r-1$ intervals.
\end{abstract}

\section{Introduction}

Given natural numbers $m$ and $n$, let $[m]^n$ be the collection of all \emph{words} of length $n$ with letters taken from the alphabet $[m] = \{1,\dots,m\}$. We write $[m]^n_{\star} = \big( [m]\cup\{\star\} \big)^n \setminus [m]^n$ and refer to the coordinates in $\bell \in [m]^n_{\star}$ where the symbol $\star$ occurs as \emph{active}. Let $\bell[\alpha]$ denote the word in $[m]^n$ obtained by substituting each occurrence of the symbol $\star$ in $\bell$ by $\alpha \in [m]$ and note that $\bell$ must contain at least one active coordinate. The set of $m$ words $\{\bell[1], \dots, \bell[m] \} \subset [m]^n$ is referred to as a \emph{combinatorial line} in $[m]^n$. We will usually simply refer to $\bell$ as that combinatorial line.

\begin{theorem}[Hales--Jewett~\cite{HJ63}]
	For any natural numbers $m$ and $r$ there exists a number $n$ such that any $r$--colouring of $[m]^n$ contains a monochromatic combinatorial line.
\end{theorem}

The smallest such $n$ is called the \emph{Hales--Jewett number} and denoted $H\!J(m,r)$. Shelah’s celebrated proof of the theorem~\cite{Sh88} uses a single induction (on $m$) and therefore gives primitive recursive bounds for the $H\!J(m,r)$. Moreover, it yields monochromatic combinatorial lines with a specific structure, which has drawn attention of several researchers. The following definitions help us discuss these results: we refer to a combinatorial line as \emph{$q$--fold} if the set of its active coordinates consists of at most $q$ sub-intervals of $\{1,\dots,n\}$ for some positive integer $q$. Let $\mI (m, r)$ be the minimum $q$ so that for sufficiently large $n$ any $r$--colouring of $[m]^n$ contains a $q$--fold combinatorial line. Shelah's argument implies that $\mI(m, r) \leq H\!J(m-1, r)$. Conlon and the first author~\cite{CK18} showed that the bound is sharp for $m=3$ and odd $r$, i.e.~$\mI(3, r) \geq r = H\!J(2, r)$. This gives Shelah's approach a certain additional weight.

Perhaps surprisingly, Leader and R\"aty~\cite{LR18} showed that the restriction on the parity of $r$ is necessary by proving that $\mI(3,2) = 1 < H\!J(2,2)$. Our goal is to show that this case is not an exception by extending their result to any even number of colours.
\begin{theorem} \label{thm:main}
	For any even $r \geq 2$ there exists $N = N(r)$ such that any $r$--colouring of $[3]^n$ contains a monochromatic combinatorial line whose set of active coordinates is contained in at most $r-1$ intervals for $n \geq N$.	
\end{theorem}
More concisely, $\mI (3, r) = r-1$ for even $r$. The fact that $\mI(3, r)$ depends on the parity of $r$ could also come as a surprise seeing as the Hales--Jewett Theorem is purely combinatorial. Several ideas behind the proof of Theorem~\ref{thm:main} will overlap with those of~\cite{CK18} and \cite{LR18} and it is perhaps advisable to first familiarise oneself with these two (significantly shorter) papers. In the following, we will sketch the idea behind our result, which requires some additional notation.

Given a word $\bw = w_1 \dots w_n \in [m]^n$, let its \emph{contraction} $\overline{\bw}$ be obtained by contracting every interval on which $\bw$ is constant to a single letter of $[m]$.
\begin{example}
	The word $\bw = 11233322$ has the contraction $\overline{\bw} = 1232$.
\end{example}
We write $\mP(m,n) = \{ \overline{\bw} : \bw \in [m]^n\}$ for the set of \emph{patterns} with alphabet $[m]$ of length at most $n$. Given some pattern $\bp = p_1 \dots p_k \in \mP(m,n)$ and $i \in [m]$, we also use the notation
\begin{equation}
	\varphi_i(\bp) = \# \{1 \leq j \leq k : p_j = i \} \quad \text{and} \quad \varphi(\bp) = \big(\varphi_1(\bp),\dots,\varphi_m(\bp)\big).
\end{equation}
We refer to $\varphi(\bp)$ as the \emph{count} of $\bp$. Lastly, we write
\begin{equation}
	\varphi^{(q+1)}_i(\bp) = \varphi_i(\bp) \pmod{ q+1} \quad \text{and} \quad \varphi^{(q+1)}(\bp) = \big(\varphi^{(q+1)}_1(\bp),\dots,\varphi^{(q+1)}_m(\bp)\big)
\end{equation}
and refer to $\varphi^{(q+1)}(\bp) \in \ZZ_{q+1}^m$ as the \emph{reduced count} of $\bp$. From now on, the number of colours will always be equal to $q+1$ and we aim to find a monochromatic $q$--fold combinatorial line.

\medskip

The colouring that Conlon and the first author constructed to show that $\mI(3, r) \geq r$ for odd $r$ is in fact a function of $\varphi^{(q+1)}(\overline{\bp})$. We will show that colourings of this type are inherent to the problem by passing precisely from any colouring of $[3]^n$ to a function of $\varphi^{(q+1)}(\overline{\bp})$.

We can describe the idea of the proof as follows: consider the $q$--fold combinatorial lines in $[3]^n$ as the hyperedges of a $3$--uniform hypergraph with vertex set $[3]^n$. Any colouring of $[3]^n$ that avoids monochromatic $q$--fold combinatorial lines simply corresponds to a proper vertex-colouring of this hypergraph, i.e.~a colouring with no monochromatic edges. Using a purely Ramsey-theoretic argument, we first show that for any such colouring and $n$ large enough we can find a sub-hypergraph isomorphic to the $q$--fold combinatorial lines in $[3]^{\tilde{n}}$, where $\tilde{n}$ is significantly smaller than $n$, with the following important property: any two words in this sub-hypergraph that have the same contraction must also have the same colour. We can therefore identify all words in this sub-hypergraph that get contracted to the same pattern.

We continue by showing that this sub-hypergraph has a rich structure, consisting of many interlaced cliques of size $q+1$, that is it contains $q$--powers of arbitrarily long paths. Besides establishing that any proper colouring of the original hypergraph requires at least $q+1$ colours, this structure will imply that, within a significant part of our sub-hypergraph, we can identify patterns with each other if they have the same reduced count $\varphi^{(q+1)}$. This establishes that any $(q+1)$--colouring of $[3]^n$ that avoids monochromatic $q$--fold combinatorial lines implies the existence of a proper $(q+1)$--colouring of a particular hypergraph with vertex set $\ZZ_{q+1}^{3}$. This hypergraph is translation-invariant, has edges between any two vertices which differ in a single coordinate as well as some important additional restrictions. These restrictions imply that it cannot be $(q+1)$--colourable for odd $q$, from which the main theorem follows.

We note that everything up to the bound on the chromatic number holds for general $q$ and that the initial Ramsey theoretic reduction to the patterns also holds for general $m$. Therefore, our proof along with the colouring from~\cite{CK18} give a good intuition on why the function $\mI(3, r)$ displays the afore-mentioned alternating behaviour.

\medskip

\noindent {\bf Outline. }  We start Section~\ref{sec:proof} by formally defining the hypergraphs just outlined in the sketch of our proof. We proceed by formalising the reduction to the patterns and then to the reduced count. The final proof ingredient is a lower bound on the chromatic number of the hypergraph defined using the reduced count. Section~\ref{sec:remarks} contains some further remarks and open questions.

\section{Proof of Theorem~\ref{thm:main}} \label{sec:proof}

Given some pattern $\bp \in \mP(3,n)$ and $k \in \NN$, the notation
\begin{equation}
	(\bp)^{k} = \overset{k \text{ times}}{\overbrace{\bp\bp \dots \bp}}
\end{equation}
refers to the $k$--fold repetition of that pattern. For the rest of the section, we fix
\begin{equation}
	k_0 = k_0(q) = 10q+6
\end{equation}
and define the \emph{buffered version} of a pattern $\bx = x_1 \dots x_{k'} \in \mP(3,k_0)$ satisfying $x_1 \neq 1$ and $x_{k'} \neq 2$ to be
\begin{equation}
	\bx^{+} = 1 \, \bx \, (23)^{2k_0} \, (13)^{2k_0} \, (21)^{2k_0} \, 231 \: \in \, \mP(3,7k_0+4).
\end{equation}
Using these definitions, we define the following hypergraphs for any $m \geq 2$. 
\begin{itemize} \setlength\itemsep{0em}
	\item	$\bH(m,n,q)$ refers to the hypergraph with vertex set $[m]^n$ and edge set consisting of all $q$--fold combinatorial lines in $[m]^n$.
	\item	$\bP(m,n,q)$ refers to the hypergraph obtained from $\bH(m,n,q)$ by identifying vertices whose contractions are the same and keeping the hyperedges. Each vertex of $\bH(m, n, q)$ is mapped to its contraction in $\bP(m, n, q)$.
	\item	$\bC(n,q)$ refers to the hypergraph on $\ZZ^3_{q+1}$ obtained by taking the sub-hypergraph of $\bP(3,n,q)$ induced by the set of vertices $\bp \in \mP(3,7k_0 + 4 + q)$ that are buffered, that is $\bp = \bx^{+}$ for some appropriate $\bx \in \mP(3,k_0)$, and then identifying and labelling vertices based on the reduced count $\varphi^{(q+1)}(\bx)$. That is, a hyperedge $\{\bx ^+, \bu ^+, \bv^+ \}$ in $\bP(3, n, q)$ induces the hyperedge  $\{\varphi^{(q+1)}(\bx), \varphi^{(q+1)}(\bu), \varphi^{(q+1)}(\bv) \}$ in $\bC(n, q)$, assuming $\bx^+,\bu^+$ and $\bv^+$ are of length at most $7k_0 + 4 + q$.
\end{itemize}

In the remainder of this section, we first show that if $\bP(m,n,q)$ is not $r$--colourable, then neither is $\bH(m,N,q)$ for some appropriately large $N$. Then we will assume that $n \geq 7k_0 + 4 + q$, set $r = q+1$ and show that if $\bC(n,q)$ is not $(q+1)$--colourable, then neither is $\bP(3,n,q)$. We conclude the proof of Theorem~\ref{thm:main} by showing that any $(q+1)$--colouring of $\bC(n,q)$ must contain a monochromatic (hyper)edge.

\subsection{The reduction to patterns}
The notation and idea behind this part are derived from the approach of Leader and R\"aty~\cite{LR18} for the specific case of $q = 1$.

We define the set of \emph{breakpoints} of a given word $\bw = w_1 \dots w_n \in [m]^n$ to be the set $T(\bw) = \{a_1, \dots, a_k\}$ for which $w_{a_{i-1}+1} = \dots = w_{a_i}$ and $w_{a_i} \neq w_{a_i+1}$ for $1 \leq i \leq k+1$ where we set $a_0 = 0$ and $a_{k+1} = n$. Let $S^{(k)}$ refer to all subsets of size $k$ of some given set $S$. Given some $N \geq n$ and $A = \{a_1 < \dots < a_{n-1} \} \in [N-1]^{(n-1)}$, let $\bw^A$ denote the word $\bw^A = w^A_1 \dots w^A_N \in [m]^N$ defined by $w^A_{a_{i-1}+1} = \dots = w^A_{a_i} = w_{i}$ for $1 \leq i \leq n$ where we set $a_0 = 0$ and $a_n = N$. Note that in general $\overline{\bw^A} = \overline{\bw}$ and $T(\bw ^A) \subseteq A$. Specifically $T(w^A) = A$ if and only if $\overline{\bw} = \bw$.
\begin{example}
	Let $\bw = 13323$ be given. We have $T(\bw) = \{1,3,4\}$. If $A = \{2,3,5,6\}$ and $N = 8$ then $\bw^A = 11 \, 3 \, 33 \,2 \, 33$.
\end{example}
 This notation allows us to make the following statement.

\begin{proposition} \label{prop:HtoP}
	For any $n \in \NN$ there exists $N = N(n) \in \NN$	so that for any colouring $\chi$ of $[m]^N$ there exists $A = A(n,N,\chi) \in [N-1]^{(n-1)}$ such that $\chi(\bw_1^A) = \chi(\bw_2^A)$ for any $\bw_1,\bw_2 \in [m]^n$ satisfying $\overline{\bw_1} = \overline{\bw_2}$.
\end{proposition}

\begin{proof}
	Give the patterns in $\mP(m,n)$ an arbitrary ordering, say $\mP(m,n) = \{ \bp_1, \dots, \bp_k \}$, and write $t_i = |\bp_i|$ for their respective length. Set $n_0 = n-1$ and recursively define $n_i = R^{(t_i-1)}(n_{i-1})$ for $1 \leq i \leq k$ where $R^{(t)}(s) = R^{(t)}(s,s)$ is the $t$--set Ramsey number. Lastly, set $N = N(n) = n_k + 1$.
	
	Let us recursively define sets $T_k \supset \dots \supset T_1 \supset T_0$ satisfying $|T_i| \geq n_i$ for $0 \leq i \leq k$ as well as certain properties with respect to the colouring $\chi$. We start by setting $T_k = [N-1]$. Let $|T_i| \geq n_i$ be given and observe that $\chi$ induces a colouring $\chi_i$ on $T_i^{(t_i-1)}$ given by $\chi_i(A) = \chi(\bp_i^A)$ for $A \in T_i^{(t_i-1)}$. Now, since by definition $n_i = R^{(t_i-1)}(n_{i-1})$, it follows that there exists $T_{i-1} \subset T_i$ satisfying $|T_{i-1}| \geq n_{i-1}$ such that $T_{i-1}^{t_i-1}$ is monochromatic with respect to $\chi_i$.
	
	Now fix some $A = \{a_1 < \dots < a_{n-1}\} \in T_0^{(n-1)}$. Consider two words $\bw_1,\bw_2 \in [m]^n$ which satisfy $\overline{\bw_1} = \overline{\bw_2} = \bp_j$ for some $1 \leq j \leq k$. We note that there exist $A_1,A_2 \in A^{(t_j-1)}$ such that $\bw_1^A = \bp_j^{A_1}$ and $\bw_2^A = \bp_j^{A_2}$, that is $T(\bw_1^A) = A_1$ and $T(\bw_2^A) = A_2$. Since $T_0 \subset T_{j-1}$, we have $A_1,A_2 \in T_{j-1}^{(t_j - 1)}$ and therefore
	\begin{equation*}
		\chi(\bw_1^A) = \chi(\bp_j^{A_1}) = \chi_j(A_1) = \chi_j(A_2) = \chi(\bp_j^{A_2}) = \chi(\bw_2^A)
	\end{equation*}
	as desired.
\end{proof}
The lemma states that within $[m]^N$ we can find a `copy' of $[m]^n$ in which any two words with the same contraction must also have the same colour. The following corollary captures this point.
\begin{corollary} \label{cor:HtoP}
	For any $n \in \NN$ there exists $N = N(n) \in \NN$ so that if $\bP(m,n,q)$ is not $r$--colourable, then neither is $\bH(m,N,q)$.	
\end{corollary}
\begin{proof}
	Given some proper colouring $\chi$ of $[m]^N$ we note that $\chi'$ given by $\chi'(\bw) = \chi(\bw^A)$ for $\bw \in [m]^n$ defines a proper colouring of $[m]^n$. This follows since any combinatorial line $\ell$ in $[m]^n$ connecting $\{\ell[1],\dots,\ell[m]\}$ corresponds to the combinatorial lines $\ell^A$ in $[m]^N$ connecting $\{\ell[1]^A,\dots,\ell[m]^A\}$. Proposition~\ref{prop:HtoP} now implies that $\chi'$ also induces a proper colouring of $\bP(m,n,q)$, proving the statement.
\end{proof}

Before we proceed to the reduction to the reduced count, let us describe the structure of the hyperedges in $\bP(m,n,q)$. Let us write $\mP_{\star}(m,n) = \{ \overline{\bell} : \bell \in [m]_{\star}^n\}$ where the contraction also contracts repeated occurrences of the symbol $\star$. We start with  a simple observation which follows from the definition of $\bP(m, n, q)$.
\begin{lemma} \label{lemma:hyperedges_P}
	For any $\bell \in \mP_{\star}(m,n)$, the set $\left \{ \overline{\bell[1]}, \overline{\bell[2]}, \dots, \overline{\bell[m]}\right \}$ forms an edge in $\bP(m,n,q)$.
\end{lemma}
\begin{proof}
	Let $\bell = \bell_1 \dots \bell_k$ where $k \leq n$. Clearly $\bell' = \bell_1 \dots \bell_k (\bell_k)^{n-k} \in [m]_{\star}^n$ is a combinatorial line in $[m]^n$ so that $\{ \overline{\bell[1]}, \overline{\bell[2]}, \dots, \overline{\bell[m]} \} = \{ \overline{\bell'[1]}, \overline{\bell'[2]}, \dots, \overline{\bell'[m]} \}$ forms an edge in $\bP(m,n,q)$.
\end{proof}
In general, the edges described in the previous lemma are of order $m$ or $m-1$. The central observation being used for the next reduction is that for the case of $m = 3$ a combinatorial line of the form $\bell = 1\star2$ \emph{connects} the patterns $12$ and $132$ by an edge of order two, since both $\bell[1]$ and $\bell[2]$ get contracted to the same pattern. This is a particularity of that alphabet order and the main reason why this approach does not easily extend to larger $m$.

Let us derive a precise description of when an edge of order two occurs between two vertices in $\bP(3,n,q)$. Edges of order two  will be sufficient in order to realise the reduction to the reduced count, but edges of order three will be crucial at the end of the section when establishing the lower bound on the chromatic number for odd $q$.

We start by introducing two more necessary notions. Let $\{\alpha_1, \alpha_2, \alpha_3\} = \{1, 2, 3 \}$ and $\bp \in \mP(3,n)$. An \emph{$\alpha_3$--insertion} in $\bp$ is the operation of inserting a copy of the letter $\alpha_3$ between an instance of $\alpha_1$ and an instance of $\alpha_2$ in $\bp$. An \emph{$\alpha_3$--alteration} of $\bp$ is the operation of moving one instance of $\alpha_3$ whose neighbours in $\bp$ are $\alpha_1$ and $\alpha_2$ to another part of $\bp$ so that its neighbouring letters are again $\alpha_1$ and $\alpha_2$.
\begin{example}
	The pattern $13212$ can be obtained from $1212$ by a $3$--insertion and from $12312$ by a $3$--alteration.
\end{example}
We will only need the notion of insertion for the remainder of this subsection, but alterations will become important in the next one.
\begin{lemma} \label{lemma:edges_P}
	Let $\bp_1, \bp_2$ be two patterns in $\mP(3,n)$ and $\alpha \in [3]$. If $\bp_2$ is obtained from $\bp_1$ by at most $q$ successive $\alpha$--insertions, then $\bp_1$ and $\bp_2$ are adjacent in $\bP(3,n,q)$.
\end{lemma}
\begin{proof}
	Assume without loss of generality that $\alpha = 3$. We will construct some $\bell \in \mP_{\star}(3,n)$ satisfying $\overline{\bell[1]} = \overline{\bell[2]} = \bp_1$ and $\overline{\bell[3]} = \bp_2$. Let $\bp_2 = p_1 \dots p_k$ and let $j_1,\dots,j_{k'}$ denote the $k' \leq q$ indices of the $3$--insertions that take one from $\bp_1$ to $\bp_2$, that is if one removes $p_{j_1},\dots,p_{j_{k'}}$ from $\bp_2$ then one obtains $\bp_1$. We now define $\bell = \ell_1 \dots \ell_k$ by
	\begin{equation}
		\ell_i = \left\{\begin{array}{ll}
			p_i & \text{for } i \in \{1,\dots,k\} \setminus \{j_1,\dots,j_{k'}\},\\
			\star & \text{for } i \in \{j_1,\dots,j_{k'}\}.
        \end{array}\right.
	\end{equation}
	It immediately follows that $\overline{\bell[1]} = \overline{\bell[1]} = \bp_1$ and $\overline{\bell[3]} = \bp_2$, so by Lemma~\ref{lemma:hyperedges_P} $\bell$ forms an edge between $\bp_1$ and $\bp_2$ in $\bP(3,n,q)$.
\end{proof}

\subsection{The reduction to the reduced count} \label{subsec:reducedcount}

Let us introduce one last additional definition. For $\{ \alpha_1,\alpha_2,\alpha_3 \} = \{1,2,3\}$ we call a pattern in $\mP(3,n)$ \emph{$\alpha_3$--diverse} if it is of length at most $n-q$ and it contains at least $q$ copies of either of the subwords $\alpha_1 \alpha_2$ or $\alpha_2 \alpha_1$.
\begin{example}
	The pattern $121$ is $3$--diverse if $q \leq 2$ and $n \geq 5$
\end{example}
Note that we are not yet restricting ourselves to buffered patterns for the following remark and the subsequent lemma.
\begin{remark} \label{rmk:diverse}
	For any $\alpha$--diverse pattern	 $\bp \in \mP(3,n)$ there exists a sequence $\bp = \bb_1, \dots, \bb_{q+1}$ in $\mP(3,n)$ so that $\bb_{i+1}$ can be obtained from $\bb_i$ by an $\alpha$--insertion for $1 \leq i \leq q$.
\end{remark}
\begin{lemma} \label{lemma:alterationsinsertions}
	Let $\chi$ be a proper $(q+1)$--colouring of $\bP(3,n,q)$, $\bp_1,\bp_2 \in \mP(3,n)$ and $\alpha \in \{1,2,3\}$. We have $\chi(\bp_1) = \chi(\bp_2)$ if either the following two cases holds:
	\begin{enumerate} \setlength\itemsep{0em}
		\item[(i)]	$\bp_2$ can be obtained from $\bp_1$ by exactly $q+1$ $\alpha$--insertions,
		\item[(ii)]	$\bp_1$ and $\bp_2$ are $\alpha$--diverse patterns and $\bp_2$ can be obtained from $\bp_1$ by an $\alpha$--alteration.
	\end{enumerate}
\end{lemma}
\begin{proof}
	Let us start with case~(i). By assumption, there exists a sequence of patterns $\bp_1 = \bb_0,\bb_1,\dots,\bb_q,\bb_{q+1} = \bp_2$ in $\mP(3,n)$ so that $\bb_{i+1}$ is obtained from $\bb_i$ by an $\alpha$--insertion for any $0 \leq i \leq q$. By Lemma~\ref{lemma:edges_P}, there is an edge connecting $\bb_i$ to $\bb_j$ if $|i-j| \leq q$ and $i \neq j$. It follows that $\{ \bb_1,\dots,\bb_q \}$ is a clique of order $q$ that lies in the neighbourhood of both $\bp_1$ and $\bp_2$. The desired statement follows.
	
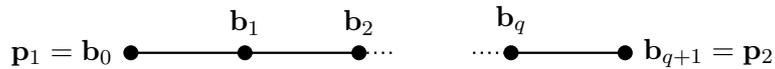
\begin{center}
	\begin{tikzpicture}
	\begin{scope}
		\draw[black, line width=0.3mm] (0,0) -- (1.5,0);
		\draw[black, line width=0.3mm] (1.5,0) -- (3,0);
		\draw[black, dotted, line width=0.3mm] (3,0) -- (3.5,0);
		\draw[black, dotted, line width=0.3mm] (4.5,0) -- (5,0);
		\draw[black, line width=0.3mm] (5,0) -- (6.5,0);
		\node[circle,fill=black,inner sep=2pt,minimum size=4pt,label=left:{$\bp_1 = \bb_0$}] (char) at (0,0) {\color{white} \scriptsize \bf };
		\node[circle,fill=black,inner sep=2pt,minimum size=4pt,label=above:{$\bb_1$}] (char) at (1.5,0) {\color{white} \scriptsize \bf };
		\node[circle,fill=black,inner sep=2pt,minimum size=4pt,label=above:{$\bb_2$}] (char) at (3,0) {\color{white} \scriptsize \bf };
		\node[circle,fill=black,inner sep=2pt,minimum size=4pt,label=above:{$\bb_q$}] (char) at (5,0) {\color{white} \scriptsize \bf };
		\node[circle,fill=black,inner sep=2pt,minimum size=4pt,label=right:{$\bb_{q+1} = \bp_2$}] (char) at (6.5,0) {\color{white} \scriptsize \bf };
	\end{scope}
	\end{tikzpicture}
	\end{center}

	Regarding case~(ii), one can easily see that if $\bp_2$ can be obtained from $\bp_1$ by an $\alpha$--alteration, then there exists a pattern $\bb_0$ so that both $\bp_1$ and $\bp_2$ can be obtained from $\bb_0$ by an $\alpha$--insertion. Since $\bp_1$ and $\bp_2$ are diverse, Remark~\ref{rmk:diverse} established that there also exist patterns $\bb_2, \bb_3, \dots, \bb_q$ in $\mP(3,n)$ so that $\bb_2$ can be obtained from both $\bp_1$ and $\bp_2$ by an $\alpha$--insertion and $\bb_{i+1}$ can be obtained from $\bb_i$ by an $\alpha$--insertion for $2 \leq i \leq q-1$.
	
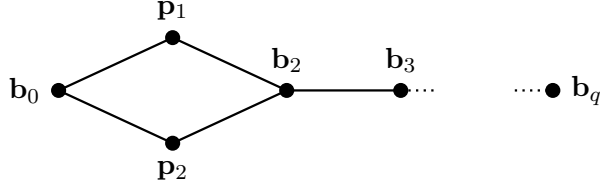
\begin{center}
	\begin{tikzpicture}
	\begin{scope}
		\draw[black, line width=0.3mm] (-1.5,0) -- (0,0.7);
		\draw[black, line width=0.3mm] (-1.5,0) -- (0,-0.7);
		\draw[black, line width=0.3mm] (0,0.7) -- (1.5,0);
		\draw[black, line width=0.3mm] (0,-0.7) -- (1.5,0);
		\draw[black, line width=0.3mm] (1.5,0) -- (3,0);
		\draw[black, dotted, line width=0.3mm] (3,0) -- (3.5,0);
		\draw[black, dotted, line width=0.3mm] (4.5,0) -- (5,0);
		\node[circle,fill=black,inner sep=2pt,minimum size=4pt,label=above:{$\bp_1$}] (char) at (0,0.7) {\color{white} \scriptsize \bf };
		\node[circle,fill=black,inner sep=2pt,minimum size=4pt,label=below:{$\bp_2$}] (char) at (0,-0.7) {\color{white} \scriptsize \bf };
		\node[circle,fill=black,inner sep=2pt,minimum size=4pt,label=left:{$\bb_0$}] (char) at (-1.5,0) {\color{white} \scriptsize \bf };
		\node[circle,fill=black,inner sep=2pt,minimum size=4pt,label=above:{$\bb_2$}] (char) at (1.5,0) {\color{white} \scriptsize \bf };
		\node[circle,fill=black,inner sep=2pt,minimum size=4pt,label=above:{$\bb_3$}] (char) at (3,0) {\color{white} \scriptsize \bf };
		\node[circle,fill=black,inner sep=2pt,minimum size=4pt,label=right:{$\bb_q$}] (char) at (5,0) {\color{white} \scriptsize \bf };
	\end{scope}
	\end{tikzpicture}
	\end{center}
	Again by Lemma~\ref{lemma:edges_P}, it follows that $\{ \bb_0,\bb_2,\dots,\bb_q \}$ form a clique of order $q$ that lies in the neighbourhood of both $\bp_1$ and $\bp_2$. The desired statement follows.
\end{proof}

Throughout the remainder of this part we will assume that $n \geq 7k_0+4+q$ and restrict ourselves to the patterns in the vertex set of $\bP(3,n,q)$ that are buffered, that is we will consider 
\begin{equation}
	\mP^+ = \{ \bp^+ : \bp = p_1,\dots,p_{k'} \in \mP(3, k_0) \text{ s.t. } p_1 \neq 2 \text{ and } p_{k'} \neq 1 \} \subset \mP(3,n-q).
\end{equation}
Note that we have chosen $n$ large enough so that this set is non-empty. Furthermore, since $k_0 = 10q + 6$ and $\mP^+ \subset \mP(3,n-q)$, every pattern contained in $\mP^+$ is clearly $\alpha$--diverse for any $\alpha \in [3]$. We can now establish the central lemma that allows us to perform the next reduction.
\begin{lemma} \label{lemma:countreduction}
	Let $\chi$ be a proper $(q+1)$--colouring of $\bP(3,n,q)$ and $\bp_1,\bp_2 \in \mP^+$. If the reduced count of the two patterns is the same, that is $\varphi^{(q+1)}(\bp_1) = \varphi^{(q+1)}(\bp_2)$, then $\chi (\bp_1) = \chi(\bp_2)$.
\end{lemma}
\begin{proof}
	We will show that for any $\bp \in \mP^+$ there exists a sequences of diverse words
	\begin{equation*}
		\bp = \bb_1, \bb_2, \dots, \bb_{k_1} = \bc_1, \bc_2, \dots, \bc_{k_2}
	\end{equation*}
	such that $\bb_{i+1}$ can be obtained from $\bb_{i}$ for $1 \leq i < k_1$ by a single alteration and $\bc_{i}$ can be obtained from $\bc_{i+1}$ for $1 \leq i < k_2$ by exactly $q+1$ $\alpha$--insertions where the $\alpha \in [3]$ is allowed to depend on each step. We will also show that $\bc_{k_2}$ is determined by the reduced count of $\bp$, so that the two sequences obtained by starting at $\bp_1$ and $\bp_2$ terminate in the same pattern. By Lemma~\ref{lemma:alterationsinsertions}, it follows from this that $\chi (\bp_1) = \chi(\bp_2)$.
	
	We start by constructing the sequence $\bb_1, \dots, \bb_{k_1}$: let us refer to a copy of any letter $\alpha_3 \in \{1,2,3\}$ in a pattern as \emph{movable} if it is positioned between a copy of $\alpha_1$ and $\alpha_2$ where as usual $\{\alpha_1,\alpha_2,\alpha_3\} = \{1,2,3\}$. In the pattern $12321$ for example, both copies of the letter $2$ are movable whereas none of the others are. We let $\bx \in \mP(3,k_0)$ be the pattern for which
	\begin{equation*}
		\bp = \bx^{+} = 1 \, \mid \, \bx \, \mid \, (2\,\llcorner\!\lrcorner\,3)^{2k_0} \, (1\,\llcorner\!\lrcorner\,3)^{2k_0} \, (2\,\llcorner\!\lrcorner\,1)^{2k_0} \, 231.
	\end{equation*}
	Here both the bar as well as the symbol $\llcorner\!\lrcorner$ only serve as a visual aid to help us with the following definitions: we will refer to the part between the two bars -- that is initially $\bx$ -- as the \emph{core} and to the part to the right of the second bar -- that is initially $(2\,\llcorner\!\lrcorner\,3)^{2k_0} \, (1\,\llcorner\!\lrcorner\,3)^{2k_0} \, (2\,\llcorner\!\lrcorner\,1)^{2k_0} \, 231$ -- as the \emph{buffer}. We will refer to the spaces between the $12$s, $13$s and $21$s in the buffer marked by the symbol $\llcorner\!\lrcorner$ as \emph{slots}.
	
	 We now obtain $\bb_{i+1}$ from $\bb_{i}$ by choosing an arbitrary movable letter from the core and moving it by an alteration into an appropriate slot in the buffer. We do so in a canonical fashion by always moving a letter to the left-most available slot. The slot itself, with the corresponding symbol, gets removed. Note that the bars stay in place throughout this process and serve as the reference point for our notions of core and buffer, even as both change in length.
	 
	 We iterate this until there are no more movable letters in the core and refer to the point at which this happens as $k_1'$. Since $\bx$ is of length at most $k_0$, we note that we haven constructed the buffer large enough to not only contain all of $\bx$, but also large enough that all of the $\bb_i$ remain $\alpha$--diverse for any $\alpha \in [3]$. If the core of $\bb_{k_1'}$ is empty, then we set $k_1 = k_1'$. If however the core of $\bb_{k_1'}$ is non-empty, then we must have $\bb_{k_1'} = 1 \, \mid \, 2121 \dots 21 \, \mid \, 2 \dots 2{\bf 3}1$. In this case, we proceed by moving the last $3$ in $\bb_{k_1'}$ in front of the core so that $\bb_{k_1' +1 } = 1 \, {\bf 3} \, \mid \, 2121 \dots 21 \, \mid \, 2 \dots 21$. We observe that we are now able to recursively move all remaining letters from the core into the buffer until we reach $\bb_{k_1 - 1} = 1 \, {\bf 3} \mid \, \mid \, 2 \dots 21$. We finish by moving the $3$ back to its original position, so that
	 \begin{equation*}
	 	 \bb_{k_1} = 1 \mid \, \mid \, (213)^{\varphi_1} \, (2\,\llcorner\!\lrcorner\,3)^{2k-\varphi_1} \, (123)^{\varphi_2} \, (1\,\llcorner\!\lrcorner\,3)^{2k-\varphi_2} \, (231)^{\varphi_3} \, (2\,\llcorner\!\lrcorner\,1)^{2k-\varphi_3} \, 231
	 \end{equation*}
	where $(\varphi_1,\varphi_2,\varphi_3) = \varphi(\bx)$ refers to the count of $\bx$.
	
	We now proceed by obtaining $\bc_{i+1}$ from $\bc_{i}$ by removing, for each step, exactly $q+1$ of either the $1$s, $2$s or $3$s from, respectively, the parts $(213)$, $(123)$ or $(231)$. We can continue to do so until we reach
	\begin{equation*}
		\bc_{k_2} = 1 \mid \, \mid \, (213)^{\varphi^{(q+1)}_1} \, (2\,\llcorner\!\lrcorner\,3)^{2k-\varphi_1^{(q+1)}} \, (123)^{\varphi_2^{(q+1)}} \, (1\,\llcorner\!\lrcorner\,3)^{2k-\varphi_2^{(q+1)}} \, (231)^{\varphi_3^{(q+1)}} \, (2\,\llcorner\!\lrcorner\,1)^{2k-\varphi_3^{(q+1)}} \, 231.
	\end{equation*}
	where $(\varphi_1^{(q+1)},\varphi_2^{(q+1)},\varphi_3^{(q+1)}) = \varphi^{(q+1)}(\bx)$. We note that $\bc_{k_2}$ only depends on the reduced count of the original core $\bx$. Since the original buffer is identical for any core, it follows that $\bc_{k_2}$ also only depends on the reduced count of $\bp$ as desired.
\end{proof}
\begin{corollary} \label{cor:PtoC}
	If $\bC(n,q)$ is not $(q+1)$--colourable, then neither is $\bP(3,n,q)$.	
\end{corollary}

We conclude this subsection by establishing the structure of some of the edges that can be found in $\bC(n,q)$. In fact, these will be essentially almost all of the edges that can be found in $\bC(n,q)$, though we do not provide a proof of this. Recall that in $\bC(n,q)$ the point associated with $\overline{\ell[\alpha]}^+ \in \mP^+ $ is labelled with the reduced count of $\overline{\ell[\alpha]}$ for any $\alpha \in [3]$. Let $\be_1 = (1,0,0)$, $\be_2 = (0,1,0)$ and $\be_3 = (0,0,1)$ denote the indicator vectors in $\NN^3$.

\begin{lemma} \label{lemma:hyperedges_C}
	Given any $\bx \in \ZZ_{q+1}^3$ as well as any $a_1,a_2,a_3 \in \ZZ$ satisfying $0 < a_1 + a_2 + a_3 \leq q$ and $a_i + a_j \geq 0$ for $i \neq j$, the set $\{ \bx + a_1 \, \be_1, \bx + a_2 \, \be_2, \bx + a_3 \, \be_3\}$ forms an edge in $\bC(n,q)$
	, where addition is modulo $q+1$.
\end{lemma}
\begin{proof}
	Write $\bx = (x_1,x_2,x_3)$ where $0 \leq x_1,x_2,x_3 \leq q$ are treated as integers. We note that we must have either $a_1,a_2,a_3 \geq 0$ or $a_{i_2}, a_{i_3} \geq -a_{i_1}>0$ for $\{i_1,i_2,i_3\} = \{1,2,3\}$. We will distinguish between these two cases.
	
	\medskip
	
	\noindent {\bf Case 1. } Assume that $a_1,a_2,a_3 \geq 0$. Consider 
	\begin{align*}
		\bell = \bell(x_1,x_2,x_3;a_1,a_2,a_3) =
		& \; (2 \star 3)^{a_1} \, (213)^{x_1} \, (23)^{q+1 - x_1 - a_1} \dots \\
		& \; (1 \star 3)^{a_2} \, (123)^{x_2} \, (13)^{q+1 - x_2 - a_2} \dots \\
		& \; (2 \star 1)^{a_3} \, (231)^{x_3} \, (21)^{q+1 - x_3 - a_3} \: \in \mP_{\star}(3,k_0).
	\end{align*}
	Here the dots merely indicate that the word is continued in the next line. It is easy to verify that 
	\begin{align*}
		\varphi^{(q+1)} \big( \overline{\bell[1]} \big) & = (x_1+a_1,x_2,x_3) = \bx + a_1 \, \be_1, 
	\end{align*}
	where all addition is modulo $q+1$. Similarly, $	\varphi^{(q+1)} \big( \overline{\bell[i]} \big)  = \bx + a_i \, \be_i
		$ for $i \in \{2, 3\}$. We note that $\ell$ is of length at most $10q + 6 = k_0$ so that $\bell^+ \in \mP_{\star}(3,n)$ and hence by Lemma~\ref{lemma:hyperedges_P} $\{ \bell^+[1], \bell^+[2], \bell^+[3] \}$ constitutes an edge in $\bP(3,n,q)$ so that $\{ \bx + a_1 \, \be_1, \bx + a_2 \, \be_2, \bx + a_3 \, \be_3 \}$ is an edge in $\bC(n,q)$. 
	
	\medskip
	
	\noindent {\bf Case 2. } Assume that $a_1 < 0$ and $a_2,a_3 \geq |a_1|$. The other cases follow likewise. Consider
	\begin{align*}
		\bell = \bell(x_1,x_2,x_3;a_1,a_2,a_3) = &
		  \; (1 \star 1)^{|a_1|} \, (213)^{x_1} \, (23)^{q+1 - x_1 + |a_1|} \dots \\
		& \; (1 \star 3)^{a_2 - |a_1|} \, (123)^{x_2} \, (13)^{q+1 - x_2 - a_2} \dots \\
		& \; (2 \star 1)^{a_3 - |a_1|} \, (231)^{x_3} \, (21)^{q+1 - x_3 - a_3} \: \in \mP_{\star}(3,k_0).
	\end{align*}
	It is again easy to verify that  $	\varphi^{(q+1)} \big( \overline{\bell[i]} \big)  = \bx + a_i \, \be_i
		$ for $i \in [3]$
	%
	%
	As before we conclude that $\{ \bx + a_1 \, \be_1, \bx + a_2 \, \be_2, \bx + a_3 \, \be_3 \}$ forms an edge in $\bC(n,q)$.
\end{proof}
The edges described in the following easy corollary form a {\lq}Latin cube{\rq}--type structure in $\bC(n,q)$, that is $\{ \bx, \bx + \be_i, \bx + 2 \, \be_i, \dots, \bx + q \, \be_i \}$ form a clique of order $q+1$ for any $\bx \in \ZZ_{q+1}^3$ and $i \in \{1,2,3\}$.
\begin{corollary} \label{cor:edges_C}
	For any $\bx \in \ZZ_{q+1}^3$, $i \in \{1,2,3\}$ and $a \in \{1,\dots,q\}$ there is a edge between the vertices $\bx$ and $\bx + a \, \be_i$ in $\bC(n,q)$ where addition is modulo $q+1$. Therefore, if $\chi$ is a proper $(q+1)$--colouring of $\bC(n,q)$, then for any $\bx$ and $\be_i$, each colour occurs exactly once in the clique $\{ \chi(\bx), \chi(\bx + \be_i), \chi(\bx + 2 \, \be_i), \dots, \chi(\bx + q \, \be_i )\}$.
\end{corollary}

\subsection{A lower bound on the chromatic number of $\bC(n,q)$}

Throughout this section we will continue to assume that $n \geq 7k_0 + 4 + q$ and simply write $\bC_q = \bC(n,q)$. The vertex set of $\bC_q$ is $\ZZ_{q+1}^3$ and the edges of $\bC_q$ that we will use are described in Lemma~\ref{lemma:hyperedges_C} and Corollary~\ref{cor:edges_C}. As already noted, $\bC_q$ contains plenty of cliques of order $q+1$ so that $\chi (\bC_q) \geq q+1$. We know that this bound is sharp for even $q$, but we wish to show the following:

\begin{proposition} \label{prop:main}
	For odd $q$ we have $\chi(\bC_q) > q+1$.
\end{proposition}

We will prove the proposition by considering a single colour, say `red' which is assumed to induce no hyperedges of $\bC_q$. We will show that the red set is determined by only two vertices and in  fact comes from the zero set of a linear functional. We show this using an inductive argument. Implicitly,~\cite{CK18} have shown that $\chi(C_q) \leq q+1$ for odd $q$ using exactly this type of colouring. 

\begin{lemma} \label{lemma:ind}
	Let $\chi$ be a proper $q+1$--colouring of $\bC_q$. Let $a  \in \ZZ$ and $b  \in \NN_0$ such that $|a | \leq b $ as well as $\max(b ,a +b ) \leq q$ and let $\{i_a ,i_b ,i_0  \}= \{1,2,3\}$. If there exists $\bx  \in \ZZ_{q+1}^3$ such that 
	\begin{equation}
		\chi(\bx ) = \chi(\bx  - a  \, \be_{i_a } + b  \, \be_{i_b })
	\end{equation}
	then we must have 
	\begin{equation} \label{eq:colour_move}
		\chi \Big( \bx  + s_0 \, \big(- (a +b ) \, e_{i_0} -b  \, e_{i_b } \big) + s_1 \, \big(-a  \, e_{i_a } + b  \, e_{i_b }\big) \Big) = \chi(\bx)
	\end{equation}
	for any $s_0, s_1 \in \ZZ$. Here all addition is modulo $q+1$.
\end{lemma}

\begin{proof}

	Let us highlight two special cases of~\eqref{eq:colour_move} that will be needed throughout the proof:
	\begin{enumerate} \setlength\itemsep{0em}
	    \item \label{item:p11} Equation~\eqref{eq:colour_move} with $(s_0, s_1) = (1, 1)$ reads $\chi \big( \bx   - (a +b ) \, \be_{i_0 } -a \, \be_{i_a } \big) = \chi(\bx )$.
		\item \label{item:p-10}Equation~\eqref{eq:colour_move} with $(s_0, s_1) = (-1, 0)$ reads $\chi \big( \bx  +  (a +b ) \, \be_{i_0 } + b \, \be_{i_b } \big) = \chi(\bx )$.
	\end{enumerate}
	%
	%
	We now prove the statement by induction on	
	\begin{equation}
		d = \max(b , a +b ) = \left\{\begin{array}{ll} a +b   & \text{if } a  \geq 0 \\ b  & \text{if } a  < 0 \end{array}\right..
	\end{equation}
	Note that by assumption $0 \leq d \leq q$. For $d = 0$ we must have $a = b = 0$, for which the statement is trivially true. We therefore assume that the inductive hypothesis holds for $0,1, \dots, d-1 < q$ and prove it for $d$. 
	
 	The case $a = 0$ immediately leads to a contradiction since $\bx$ and $\bx + b \, \be_{i_b}$ are adjacent by Corollary~\ref{cor:edges_C}. The case where $a+b = 0$, that is $a = -d$ and $b = d$, will be argued separately at the end as it relies on previously having proven this statement for all other cases. Let us therefore look at the case $a+b > 0$. We start by proving the special case $(s_0,s_1) = (1,1)$ of equation~\eqref{eq:colour_move}.
	\begin{claim} \label{claim:zred}
	    Let $a' \in \ZZ \setminus \{0\}$ and $b'\in \NN$ satisfy $b' \geq a'> -b'$ as well as $\max(b', a' +b') = d$. If there are $\{ i_a',i_b',i_0' \}= \{1,2,3\}$ and $\bx' \in \ZZ_{q+1}^3$ such that
    	\begin{equation}
    		\chi(\bx') = \chi(\bx' - a'\, \be_{i_a'} + b'\, \be_{i_b'}),
    	\end{equation}
		then
    	\begin{equation}
    	    \chi(\bx' - a'\, \be_{i_a'} - (a' +b')\,\be_{i_0'}) = \chi (\bx').
    	\end{equation}
	\end{claim}
    \begin{proof}[Proof of Claim~\ref{claim:zred}]
    	For the remainder of the proof we simply say that the colour of $\bx'$ is red. We write 
		\begin{align*}
			\bc = \bx' - a'\, \be_{i_a'}, \quad \by' = \bc + b'\, \be_{i_b'}, \quad \text{and} \quad \bz = \bc - (a' +b') \, \be_{i_0'}.
		\end{align*}
		Recall that the aim is to show that $\bz$ is red. Before proceeding with a case distinction, we note that $\bc$ cannot be red since it is adjacent to both $\bx'$ and $\by'$ by Corollary~\ref{cor:edges_C}

		\medskip \noindent {\bf Case~1.} Assume that $a'> 0$. Let us show that out of the vertices $\bc + j \, \be_{i_0'}$ where $-(a'+b') \leq j \leq q-(a'+b')$ only $\bz$,  corresponding to $j = -(a'+b') $, \emph{can} be red and therefore in fact \emph{must} be red by Corollary~\ref{cor:edges_C}. We will do so through a case distinction illustrated in Figure~\ref{fig:case1}.
		
		\begin{figure}[h]
		\begin{center}
		\scalebox{1.15}{
		\begin{tikzpicture}
			\begin{scope}
				\draw [thick, decoration={brace, mirror, raise=0.3cm }, decorate] (0.02,0) -- (4.98,0) node [pos=0.5,anchor=north,yshift=-0.4cm] {\footnotesize Case 1.1};
				\draw [thick, decoration={brace, mirror, raise=0.3cm }, decorate] (-1.98,0) -- (-0.02,0) node [pos=0.5,anchor=north,yshift=-0.4cm] {\footnotesize Case 1.3}; 
				\draw [thick, decoration={brace, mirror, raise=0.3cm }, decorate] (-2.98,0) -- (-2.02,0) node [pos=0.5,anchor=north,yshift=-0.4cm] {\footnotesize Case 1.2}; 
				\draw[black, line width=0.4mm] (0,0) -- (0.71,0.71);
				\draw[black, line width=0.4mm] (0,0) -- (0,2);
				\draw[black, line width=0.4mm] (0,0) -- (-3,0);
				\draw[black, dotted, line width=0.4mm] (-4,0) -- (-3,0);
				\draw[black, dotted, line width=0.4mm] (5,0) -- (6,0);
				\draw[black, line width=0.4mm] (0,0) -- (5,0);
				\node[inner sep=0pt,minimum size=5pt,label={\tiny $a'$}] (a) at (0.3,0.3) {};
				\node[inner sep=0pt,minimum size=5pt,label={\tiny $b'$}] (a) at (-0.2,0.9) {};
				\node[inner sep=0pt,minimum size=5pt,label=above:{\tiny $-(a'+b')$}] (a) at (-1.7,-0.1) {};
				
				\node[inner sep=0pt,minimum size=5pt,label=above:{\tiny $(q+1)-(a'+b')$}] (a) at (2.7,-0.1) {};
				\node[circle,fill=black,inner sep=0pt,minimum size=5pt,label={[xshift=-0.3cm]$\bc$}] (a) at (0,0) {};
				\node[circle,fill=black,inner sep=0pt,minimum size=5pt,label={[xshift=0.3cm]$\bx'$}] (a) at (0.71,0.71) {};
				\node[circle,fill=black,inner sep=0pt,minimum size=5pt,label={[xshift=0.3cm]$\by'$}] (a) at (0,2) {};
				\node[circle,fill=black,inner sep=0pt,minimum size=5pt,label=above:{$\bz$}] (a) at (-3,0) {};
				\node[circle,fill=black,inner sep=0pt,minimum size=5pt,label=above:{$\bz$}] (a) at (5,0) {};
			\end{scope}
		\end{tikzpicture}}
		\end{center}
		\caption{Relative positions of the involved points in the case $a' > 0$.} \label{fig:case1}
		\end{figure}
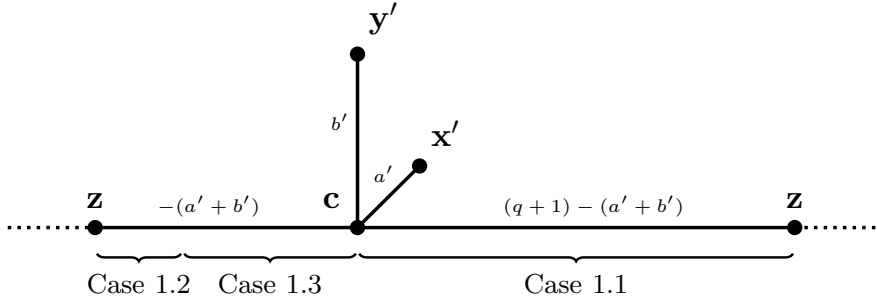
		
		\smallskip \noindent \emph{Case~1.1.} If $1 \leq j < (q+1) - (a' +b')$, then we note that $\bx'$ and $\by'$ form a hyperedge with any of the vertices $\bc + j \, \be_{i_0'}$ as described in Lemma~\ref{lemma:hyperedges_C} and therefore these vertices cannot be red.
		
		\smallskip \noindent \emph{Case~1.2.} If $- (a' +b') < j < -b'$, then  using \ref{item:p-10} with
		\begin{equation}
			a = -a', \quad b  = -j, \quad i_{a } = i_a', \quad i_{b } = i_0' \quad \text{and} \quad \bx  = \bc + j \, \be_{i_0'},
		\end{equation}
		we get that $\bc + (a +b ) \, \be_{i_b'}$ is red.  Note that we could use the inductive hypothesis as $b = - j> b' \geq |a'| = |a|$ and $\max(b ,a +b ) = b = -j < a' +b' = d$. Since $\bx' = \bc + b'\, \be_{i_b'}$ is also red, we have $b' = a+b = - j - a'$ by Corollary~\ref{cor:edges_C}, implying the contradiction $a' + b'=  - j < a' +b'$.
		
		\smallskip \noindent \emph{Case~1.3.} If $-b' \leq j \leq -1$, then using the inductive hypothesis in the form of~\ref{item:p11} with
		\begin{equation}
			a  = j, \quad b  = b', \quad i_{a } = i_0', \quad i_{b } = i_b'\quad \text{and} \quad \bx  = \bc + j \, \be_{i_0'},
		\end{equation}
		we get that $\bc - (a +b ) \, \be_{i_a'}$ is red. Note that we could use the inductive hypothesis as $b = b' \geq - j = |a|$ and $\max(b ,a +b ) = b'< a' + b'= d$. However, since $\bc + a'\,\be_{i_a'}$ is also red, it follows that $a' = - (a+b) = - j  - b'$ by Corollary~\ref{cor:edges_C}, giving the contradiction $a' +b'=  - j  < a' +b'$.
		
		\medskip \noindent {\bf Case~2.} Assume that $a'< 0$. Let us again show that out of the vertices $\bc + j \, \be_{i_0'}$ where $-(a'+b') \leq j \leq q-(a'+b')$ only $\bz$, corresponding to $j =-(a'+b') $, is red. We will do so through another case distinction that is illustrated in Figure~\ref{fig:case2}.
	
		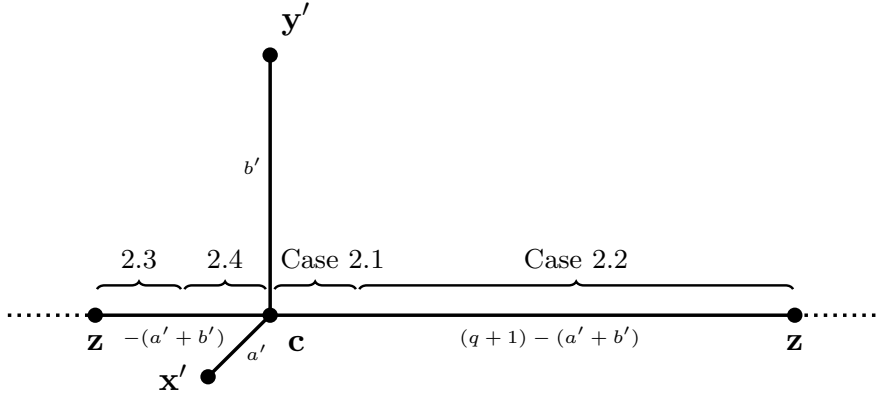
\begin{figure}[h]
		\begin{center}
		\scalebox{1.15}{
		\begin{tikzpicture}
			\begin{scope}
				\draw [thick, decoration={brace, raise=0.3cm }, decorate] (1.02,0) -- (5.98,0) node [pos=0.5,anchor=south,yshift=0.4cm] {\footnotesize Case 2.2};
				\draw [thick, decoration={brace, raise=0.3cm }, decorate] (0.05,0) -- (0.98,0) node [pos=0.5,anchor=south,yshift=0.4cm,xshift=0.2cm] {\footnotesize Case 2.1};
				\draw [thick, decoration={brace, raise=0.3cm }, decorate] (-0.98,0) -- (-0.05,0) node [pos=0.5,anchor=south,yshift=0.4cm] {\footnotesize 2.4};
				\draw [thick, decoration={brace, raise=0.3cm }, decorate] (-1.98,0) -- (-1.02,0) node [pos=0.5,anchor=south,yshift=0.4cm] {\footnotesize 2.3};
				\draw[black, line width=0.4mm] (0,0) -- (-0.71,-0.71);
				\draw[black, line width=0.4mm] (0,0) -- (0,3);
				\draw[black, line width=0.4mm] (0,0) -- (-2,0);
				\draw[black, dotted, line width=0.4mm] (-3,0) -- (-2,0);
				\draw[black, dotted, line width=0.4mm] (6,0) -- (7,0);
				\draw[black, line width=0.4mm] (0,0) -- (6,0);
				\node[inner sep=0pt,minimum size=5pt,label={\tiny $a'$}] (a) at (-0.15,-0.75) {};
				\node[inner sep=0pt,minimum size=5pt,label={\tiny $b'$}] (a) at (-0.2,1.4) {};
				\node[inner sep=0pt,minimum size=5pt,label=below:{\tiny $-(a'+b')$}] (a) at (-1.1,0.1) {};
				\node[inner sep=0pt,minimum size=5pt,label=below:{\tiny $(q+1)-(a'+b')$}] (a) at (3.2,0.1) {};
				\node[circle,fill=black,inner sep=0pt,minimum size=5pt,label={[xshift=0.3cm,yshift=-0.64cm]$\bc$}] (a) at (0,0) {};
				\node[circle,fill=black,inner sep=0pt,minimum size=5pt,label=left:{$\bx'$}] (a) at (-0.71,-0.71) {};
				\node[circle,fill=black,inner sep=0pt,minimum size=5pt,label={[xshift=0.3cm]$\by'$}] (a) at (0,3) {};
				\node[circle,fill=black,inner sep=0pt,minimum size=5pt,label=below:{$\bz$}] (a) at (-2,0) {};
				\node[circle,fill=black,inner sep=0pt,minimum size=5pt,label=below:{$\bz$}] (a) at (6,0) {};
			\end{scope}
		\end{tikzpicture}}
		\end{center}
		\caption{Relative positions of the involved points in the case $a' < 0$.} \label{fig:case2}
		\end{figure}
		
		\smallskip \noindent \emph{Case~2.1.} If $1 \leq j \leq |a'|-1$, then using \ref{item:p-10} with 
		\begin{equation}
			a  = -j, \quad b  = |a'|, \quad i_{a } = i_0', \quad i_{b } = i_a'\quad \text{and} \quad \bx  = \bx',
		\end{equation}
		we get that $\bc + (a +b ) \, \be_{i_b'}$ is red. Note that we could use the inductive hypothesis as $b = |a'| > j = |a|$ and $\max(b ,a +b ) = |a'| < b'= d$. Since $\bx' = \bc+\,b'\be_{i_b'}$ is also red, we have $b' = (a +b ) = -j + |a'|$, implying the contradiction $a' +b'= -j$.
		
		\smallskip \noindent \emph{Case~2.2.} If $|a'| \leq j < (q+1) - (a'+b)$, then we argue as in case~1.1. That is, we observe that $\bx'$ and $\by'$ form a hyperedge with any of the vertices $\bc + j \, \be_{i_0'}$ as described in Lemma~\ref{lemma:hyperedges_C} and therefore these vertices cannot be red.

		\smallskip \noindent \emph{Case~2.3.} If $-(a' +b') < j \leq -|a'|$, then we proceed as in case~1.2. That is, we apply~\ref{item:p-10} with
		\begin{equation}
			a  = |a'|, \quad b  =  - j, \quad i_{a } = i_a', \quad i_{b } = i_0' \quad \text{and} \quad \bx  = \bc + j \, \be_{i_0'},
		\end{equation}
		to conclude that $\bc + (a +b ) \, \be_{i_b'}$ is red. Note that we could use the inductive hypothesis as $b = - j \geq |a'| = |a|$ as well as $\max(b ,a +b ) = |a'|  - j < |a'| + a'+ b'= b'= d$. It follows that $b' = a+b = |a'| - j$, implying the contradiction $a' +b'= - j < a' +b'$.

		\smallskip \noindent \emph{Case~2.4.} If $\max\big(-|a'|, -(a' +b')\big) < j \leq -1$, then using \ref{item:p-10} with
		\begin{equation}
			a  =  - j, \quad b  = |a'|, \quad i_{a } = i_0', \quad i_{b } = i_a'\quad \text{and} \quad \bx  = \bx',
		\end{equation}
		we get that $\bc + (a +b ) \, \be_{i_b'}$ is red. Note that we could use the inductive hypothesis as $b = |a'| > - j = |a|$ as well as $ \max(b , a +b ) = |a'|  - j < |a'| + a' +b' = b'= d$. Hence $b' = a+b =  - j + |a'|$, implying the contradiction $a' +b'= - j < a' +b'$.

		\smallskip As previously in Case~1, we conclude that $\bz$ must be red. This concludes the proof of Claim~\ref{claim:zred}.
	\renewcommand{\qedsymbol}{$\blacksquare$}
	\end{proof}
	
	Claim~\ref{claim:zred} will now be used to show~\eqref{eq:colour_move} in full generality, so for all $s_0, s_1 \in \ZZ$, when $a+b >0$ and $\max(b, a+b) = d$. For this purpose, let $\bx, i_a, i_b$ and $i_0$ be as stated in the lemma and recall that
	\begin{equation*} \label{eq:pyramid_scheme}
	    \bp(s_0,s_1) = \bx + s_0 \, \big( -(a+b) \,\be_{i_0} - b \,\be_{i_b} \big) + s_1 \, \big( -a \,\be_{i_a} + b \,\be_{i_b} \big).
	\end{equation*}
	Again, we will say that $\bp(0,0) = \bx$ and  $\bp(0,1) = \bx - a \,\be_{i_a} + b\,\be_{i_b}$ are red. As previously in the proof of the Claim~\ref{claim:zred}, we have to distinguish two cases for $a$.
		
	\begin{figure}[h]
	\begin{center}
	\scalebox{1.3}{
	\begin{tikzpicture}
		\begin{scope}
			\coordinate (A) at (-0.5,0.3);
			\coordinate (B) at (1,0.5);
			\coordinate (AB) at (0,-1.8);
			\coordinate (C1) at (0,0);
			\coordinate (C2) at ($(C1)-(A)+(B)$);
			\coordinate (C3) at ($(C1)-3*(A)+3*(B)$);
			\coordinate (COMP) at ($(C3)+2*(AB)-3*(A)+2.1*(B)$);
			\draw[black, ->, line width=0.2mm] (COMP) -- ($(COMP)-0.5*(AB)$);
			\draw[black, ->, line width=0.2mm] (COMP) -- ($(COMP)+0.5*(A)$);
			\draw[black, ->, line width=0.2mm] (COMP) -- ($(COMP)+0.5*(B)$);
			\node[label={[xshift=-5pt,yshift=-7.3pt] \tiny $\miniplus a$}] (char) at ($(COMP)+0.5*(A)$) {};
			\node[label={[xshift=6pt,yshift=-10pt] \tiny $\miniplus b$}] (char) at ($(COMP)+0.5*(B)$) {};
			\node[label={[xshift=-1pt,yshift=-5.5pt] \tiny $\miniplus (a\miniplus b)$}] (char) at ($(COMP)-0.5*(AB)$) {};
			\draw[black!60, line width=0.3mm] ($(C1)+(AB)+(A)$) -- ($(C1)+(AB)$);
			\draw[black!60, line width=0.3mm] ($(C1)+(AB)+(A)$) -- ($(C1)+(AB)+(A)-(B)$);
			\draw[black!60, line width=0.3mm] ($(C1)+(AB)+(A)$) -- ($(C1)+(A)$);
			\draw[black!60, line width=0.3mm] ($(C3)+(AB)+(B)$) -- ($(C3)+(AB)$);
			\draw[black!60, line width=0.3mm] ($(C3)+(AB)+(B)$) -- ($(C3)+(AB)+(B)-(A)$);
			\draw[black!60, line width=0.3mm] ($(C3)+(AB)+(B)$) -- ($(C3)+(B)$);
			\draw[black, line width=0.3mm] (C1) -- ($(C1)+(AB)$);
			\draw[black, line width=0.3mm] (C1) -- ($(C1)+(A)$);
			\draw[black, line width=0.3mm] (C1) -- ($(C1)+(B)$);
			\draw[black, line width=0.3mm] (C2) -- ($(C2)+(AB)$);
			\draw[black, line width=0.3mm] (C2) -- ($(C2)+(A)$);
			\draw[black, line width=0.3mm] (C2) -- ($(C2)+(B)$);
			\draw[black, line width=0.3mm] (C3) -- ($(C3)+(AB)$);
			\draw[black, line width=0.3mm] (C3) -- ($(C3)+(A)$);
			\draw[black, line width=0.3mm] (C3) -- ($(C3)+(B)$);
			\draw[black, dotted, line width=0.2mm] ($(C1)+(AB)-(A)$) -- ($(C1)+(AB)-(A)+0.33*(AB)$);
			\draw[black, dotted, line width=0.2mm] ($(C1)+(AB)-(A)$) -- ($(C1)+(AB)-(A)+(A)$);
			\draw[black, dotted, line width=0.2mm] ($(C1)+(AB)-(A)$) -- ($(C1)+(AB)-(A)+(B)$);
			\draw[black, dotted, line width=0.2mm] ($(C1)+(AB)-(B)$) -- ($(C1)+(AB)-(B)+0.33*(AB)$);
			\draw[black, dotted, line width=0.2mm] ($(C1)+(AB)-(B)$) -- ($(C1)+(AB)-(B)+(A)$);
			\draw[black, dotted, line width=0.2mm] ($(C1)+(AB)-(B)$) -- ($(C1)+(AB)-(B)+(B)$);
			\draw[black, dotted, line width=0.2mm] ($(C3)+(AB)-(A)$) -- ($(C3)+(AB)-(A)+0.33*(AB)$);
			\draw[black, dotted, line width=0.2mm] ($(C3)+(AB)-(A)$) -- ($(C3)+(AB)-(A)+(A)$);
			\draw[black, dotted, line width=0.2mm] ($(C3)+(AB)-(A)$) -- ($(C3)+(AB)-(A)+(B)$);
			\draw[black!60, dotted, line width=0.2mm] ($(C1)+(AB)+(A)-(B)$) -- ($(C1)+(AB)+(A)-(B)+0.33*(AB)$);
			\draw[black!60, dotted, line width=0.2mm] ($(C3)+(AB)+(B)-(A)$) -- ($(C3)+(AB)+(B)-(A)+0.33*(AB)$);
			\draw[black!60, dotted, line width=0.2mm] ($(C1)+(A)$) -- ($(C1)+(A)+(B)$);
			\draw[black!60, dotted, line width=0.2mm] ($(C1)+(B)$) -- ($(C1)+(A)+(B)$);
			\draw[black!60, dotted, line width=0.2mm] ($(C3)+(A)$) -- ($(C3)+(A)+(B)$);
			\draw[black!60, dotted, line width=0.2mm] ($(C3)+(B)$) -- ($(C3)+(A)+(B)$);
			\draw[black!60, dotted, line width=0.2mm] ($(C1)+(A)+(B)$) -- ($(C1)+(A)+(B)-0.33*(AB)$);
			\draw[black!60, dotted, line width=0.2mm] ($(C3)+(A)+(B)$) -- ($(C3)+(A)+(B)-0.33*(AB)$);
			\draw[black, dotted, line width=0.2mm] ($(C1)+(B)$) -- ($(C1)+(B)-0.33*(AB)$);
			\draw[black, dotted, line width=0.2mm] ($(C2)+(B)$) -- ($(C2)+(B)-0.33*(AB)$);
			\draw[black, dotted, line width=0.2mm] ($(C3)+(A)$) -- ($(C3)+(A)-0.33*(AB)$);
			\draw[black, dotted, line width=0.2mm] ($(C2)+(B)$) -- ($(C2)+(B)-0.5*(A)$);
			\draw[black, dotted, line width=0.2mm] ($(C3)+(A)$) -- ($(C3)+(A)-0.3*(B)$);
			\draw[black, dotted, line width=0.2mm] ($(C2)+(AB)$) -- ($(C2)+(AB)-(A)$);
			\draw[black, dotted, line width=0.2mm] ($(C2)+(AB)-(A)$) -- ($(C2)+(AB)-(A)+0.8*(B)$);
			\draw[black, dotted, line width=0.2mm] ($(C2)+(AB)-(A)$) -- ($(C2)+(AB)-(A)+0.33*(AB)$);
			\draw[black, dotted, line width=0.2mm] ($(C3)+(AB)$) -- ($(C3)+(AB)-0.8*(B)$);
			\node[circle,fill=black,inner sep=0pt,minimum size=4pt,label=left:{\tiny $\bp(s_0\miniminus 1,\!0)$}] (char) at ($(C1)+(A)$) {\color{white} \tiny };
			\node[circle,fill=black,inner sep=0pt,minimum size=4pt,label={[xshift=17pt,yshift=-3pt] \tiny $\bp(s_0\miniminus 1,\!1)$}] (char) at ($(C2)+(A)$) {\color{white} \tiny };
			\node[circle,fill=black,inner sep=0pt,minimum size=4pt,label=above:{\tiny }] (char) at ($(C2)+(B)$) {\color{white} \tiny };
			\node[circle,fill=black,inner sep=0pt,minimum size=4pt,label=above:{\tiny }] (char) at ($(C3)+(A)$) {\color{white} \tiny };
			\node[circle,fill=black,inner sep=0pt,minimum size=4pt,label=right:{\tiny $\bp(s_0\miniminus 1,\!s_0)$}] (char) at ($(C3)+(B)$) {\color{white} \tiny };
			\node[circle,fill=black,inner sep=2pt,minimum size=0pt,label={[xshift=16pt,yshift=-3pt]\tiny $\bp(s_0,\!1)$}] (char) at ($(C1)+(AB)$) {\color{white} \scriptsize \bf 1};
			\node[circle,fill=black,inner sep=2pt,minimum size=0pt,label={[xshift=16pt,yshift=-3pt]\tiny }] (char) at ($(C2)+(AB)$) {\color{white} \scriptsize \bf 1};
			\node[circle,fill=black,inner sep=2pt,minimum size=0pt,label={[xshift=-18pt,yshift=-4pt]\tiny $\bp(s_0,\!s_0)$}] (char) at ($(C3)+(AB)$) {\color{white} \scriptsize \bf 1};
			\node[circle,fill=black,inner sep=2pt,minimum size=0pt,label=left:{\color{black} \tiny $\bp(s_0,\!0)$}] (char) at ($(C1)+(AB)+(A)-(B)$) {\color{white} \scriptsize \bf 2};
			\node[circle,fill=black,inner sep=2pt,minimum size=0pt,label=right:{\color{black} \tiny $\bp(s_0,\!s_0\miniplus 1)$}] (char) at ($(C3)+(AB)+(B)-(A)$) {\color{white} \scriptsize \bf 3};
		\end{scope}
	\end{tikzpicture}}
	\end{center}
	\caption{The inductive step over $s_0$ in the case $a > 0$.} \label{fig:caseI}
	\end{figure}
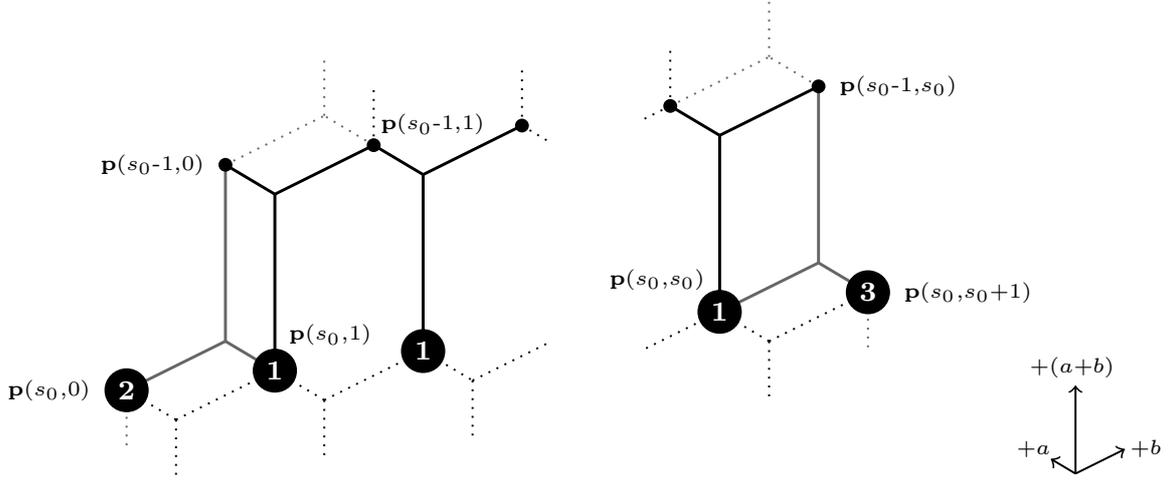
		
	\smallskip \noindent {\bf Case~I.} If $a > 0$, then we start by showing that $ \bp(s_0,s_1)$ is red for any $s_0,s_1 \in \NN_0$ satisfying $s_1 \leq s_0 + 1$ by induction on $s_0$. We note that for $s_0 = 0$ the statement holds as $\bp(0,0) = \bx$ and $\bp(0,1) = \bx - a \,\be_{i_a} + b\,\be_{i_b}$. Assume therefore that the statement holds for $s_0 - 1$ and let us show that it holds for all $0 \leq s_1 \leq s_0+1$. We will do so through another case distinction that is illustrated in Figure~\ref{fig:caseI}:
	\begin{itemize} \setlength\itemsep{0em}
		\item[1.]	 We start by proving it for $0 < s_1 < s_0 + 1$. By inductive assumption, $\bp(s_0-1,s_1-1)$ and $\bp(s_0-1,s_1)$ are red so that we can apply Claim~\ref{claim:zred} with
			\begin{equation}
				a' = a, \quad b' = b, \quad i_a' = i_a, \quad i_b' = i_b\quad \text{and} \quad \bx' = \bp(s_0-1, s_1-1),
			\end{equation}
			to deduce that $\bp(s_0, s_1)$ is also red.
		\item[2.]	Secondly, let $s_1 = 0$. Since $\bp(s_0-1, 0)$ and $\bp(s_0, 1)$ are red, we apply Claim~\ref{claim:zred} with
			\begin{equation}
				a' = -a, \quad b' = a+b, \quad i_a' = i_a, \quad i_b' = i_0 \quad \text{and} \quad \bx' = \bp(s_0, 1),
			\end{equation}
			to deduce that $\bp(s_0, 0)$ is also red.
		\item[3.]	Finally, consider $s_1 = s_0+1$. We use the fact that $\bp(s_0, s_0)$ and $\bp(s_0-1, s_0)$ are red and apply Claim~\ref{claim:zred} with
			\begin{equation}
				a' = -b, \quad b' = a+b, \quad i_a' = i_b, \quad i_b' = i_0 \quad \text{and} \quad \bx' = \bp(s_0, s_0),
			\end{equation}
			to deduce that $\bp(s_0, s_0+1)$ is also red.
	\end{itemize}
	This concludes the inductive step, so we have shown that~\eqref{eq:colour_move} holds whenever $s_0 \in \NN_0$ and $0 \leq s_1 \leq s_0 + 1$. However, since the summation is done modulo $(q+1)$, we have for instance $\bp(s_0, s_1) = \bp(s_0-(q+1), s_1)$. It immediately follows that $\bp(s_0,s_1)$ is red for any $s_1, s_0 \in \ZZ$.
	
	\smallskip \noindent {\bf Case~II.} If $a < 0$, then we note that, by Claim~\ref{claim:zred}, $\bz = \bx - a\,\be_{i_a} - (a+b)\,\be_{i_0}$ is red. If $a+b \geq |a|$, then we rewrite that last equation as $\bx = \bz - |a| \, e_{i_a} + (a+b) \, \be_{i_0}$ and observe that the position of $\bx$ in relation to $\bz$ satisfies the conditions of Case~I. If $a+b \leq |a|$ then we note that $\bz = \bx - (a+b)\,\be_{i_0} + |a|\,\be_{i_a}$ so that now $\bz$ in relation to $\bx$ satisfies the conditions of Case~I. In either scenario we can immediately derive~\eqref{eq:colour_move}.

    \smallskip This completes the inductive step over $d$ for $a+b > 0$. The remaining case is $-a = b = d$, so let $\bx$ and $\by = \bx + b\, \be_{i_a} + b\, \be_{i_b}$ be red. By Corollary~\ref{cor:edges_C}, the vertex $\bx + b\, \be_{i_a}$ is connected to both $\bx$ and $\by$ by an edge and therefore cannot be red. The vertices $\bx + b\, \be_{i_a} + j \, \be_{i_0}$ where $|a| = b \leq j \leq q$ form a hyperedge with $\bx$ and $\by$ as described in Lemma~\ref{lemma:hyperedges_C} and therefore these vertices also cannot be red. Now assume that $\bx + b\,\be_{i_a} + j_0 \, \be_{i_0}$ is red for some $1 \leq j_0 \leq b-1$. We apply~\eqref{eq:colour_move} with $-j_0$ in place of $a$ and the same value of $b$. Note that $|-j_0| = j_0 < b$ as well as $ \max(b, -j_0+b) = b = d$ and $-j_0 \neq b$. We previously established that~\eqref{eq:colour_move} holds in this case, so using observation~\ref{item:p-10} we get that $\bx +b \,\be_{i_a} + (-j_0 + b) \, \be_{i_b}$ must be red. Since $\by = \bx +b \,\be_{i_a} + b \, \be_{i_b}$ is also red, we must have $b - j_0 = b$ in contradiction to $j_0 \geq 1$. We have shown that the vertices $\bc + j \, \be_{i_0}$ cannot be red for any $0 \leq j \leq q$, contradicting Corollary~\ref{cor:edges_C}. It follows that~\eqref{eq:colour_move} vacuously holds for the case $-a = b = d$, completing the inductive step over $d$ and proving Lemma~\ref{lemma:ind}.
\end{proof}

Let us now derive Proposition~\ref{prop:main} from Lemma~\ref{lemma:ind}.
\begin{proof}[Proof of Proposition~\ref{prop:main}]
    By Corollary~\ref{cor:edges_C}, one of the vertices $\{ j \, \be_2: 0 \leq j \leq q\}$ must have the same colour as $-\be_1$, say $\chi(j_0 \,\be_2) = \chi (-\be_1)$. If $j_0 = 0$, then we get an immediate contradiction. We claim that
    \begin{equation} \label{eq:g1}
	    1 = \gcd(j_0,q+1) = \gcd(j_0-1, q+1).
    \end{equation}
    From this one would immediately derive a contradiction since $q+1$ is even by assumption and at least one of $j_0$ and $j_0-1$ must be even as well.
    
    To see that~\eqref{eq:g1} holds, we apply Lemma~\ref{lemma:ind} with
	\begin{equation}
		a = -1, \quad b = j_0, \quad i_a = 1, \quad i_b = 2 \quad \text{and} \quad \bx = -\be_1.
	\end{equation}
	Let $\bp(s_0, s_1)$ be defined as in the previous proof so that Lemma~\ref{lemma:ind} established that all $\bp(s_0,s_1)$ have the same colour as $\be_1$ and $b \, \be_2$ for any $s_0, s_1 \in \ZZ$. Let $A_a, A_b, A_0$ respectively denote the set of their projections onto the axes $i_a, i_b, i_0$. We note that
    \begin{align*}
    	|A_a| & = (q+1)/\gcd(1,q+1) = 1, \\
    	|A_b| & = (q+1)/\gcd(j_0,q+1), \\
    	|A_0| & = (q-1)/\gcd(j_0-1,q+1).
    \end{align*}
	However, in order not to violate the latin cube structure described by Corollary~\ref{cor:edges_C}, we must have $|A_a| = |A_b| = |A_0|$. This establishes~\eqref{eq:g1} and therefore concludes the proof of Proposition~\ref{prop:main}.
\end{proof}
Theorem~\ref{thm:main} now follows as an immediate consequence of Proposition~\ref{prop:main}, Corollary~\ref{cor:PtoC} and Corollary~\ref{cor:HtoP}, where we set the $N_{\footnotesize\ref{thm:main}} = N_{\footnotesize\ref{thm:main}}(r)$ of Theorem~\ref{thm:main} equal to $N_{\footnotesize\ref{cor:HtoP}}(7k_0 + 4 + q) = N_{\footnotesize\ref{cor:HtoP}}(70q + 42 + q)$ from Corollary~\ref{cor:HtoP} and where of course $q = r-1$.\qed

\section{Remarks and Open Questions} \label{sec:remarks}

\subsection{Bounds for larger alphabets}

Based on their result, Conlon and the first author originally conjectured that $\mI(m,r) = H\!J(m-1,r)$. Given the result of Leader and R\"aty, this conjecture was later retracted, though it is still reasonable to wonder if $\mI(m,r)$ is at all linked to $H\!J(m-1,r)$ when $m > 3$.
\begin{question}
	Improve on either of the immediate bounds
		\begin{equation}
				\max \left\{ \mI(m-1,r), \, \mI(m,r-1) \right\} \leq \mI(m,r) \leq H\!J(m-1,r).
		\end{equation}
\end{question}
In particular, $r$ is the best explicit lower bound that the present methods can give on $\mI(m, r)$, which is probably far from the truth. 

An upper bound separating $\mI(m, r)$ from $H\!J(m-1, r)$ would entail a new argument for the Hales-Jewett theorem. For instance, the reduction to the hypergraph $\bP(3,n, r)$ conceived by Leader and R\"aty, along with the observation that $\bP(3, n, r)$ contains an $r$-clique, already gives an alternative proof for an alphabet of size $3$. For $m \geq 3$, this strategy essentially reduces the alphabet size by one and therefore may be a good starting point for further inquiries.

\subsection{Bounds for the Hales--Jewett number}

There is currently a large gap in the best bounds on $H\!J(3, r)$, see ~\cite{Sh88,Be68,HT14,La16,BCT18,Co18}. More specifically, we have
\begin{equation}
	r^{c \ln(r)} \leq H\!J(3, r) \leq 2^{2^{cr}}
\end{equation}
for some $c > 0$, where the upper bound is due to Conlon~\cite{Co18} and the lower bound follows from a lower bound of the van der Waerden number $W(3,r)$ due to Graham, Rothschild and~Spencer~\cite{GRS90}. With this in mind, we ask the following question.
\begin{question}
	Given $r \geq 2$ and $q \geq 2 \ceil{r/2}-1$, what $N = N(r,q)$ guarantees the existence of $q$--fold lines in any $r$--colouring of $[3]^N$?
\end{question}
A better lower bound on $N(r, q)$ may turn out to be more accessible than one for $H\!J(3, r)$.

\subsection{The cyclic setting}

It should also be noted that the setup become significantly more natural if we take the cycle ground set for the coordinates, that is $\ZZ_n$ rather than $[n]$, so that we might have an interval at the border that ‘wraps’ around. In particular, the step to the reduced count in Subsection~\ref{subsec:reducedcount} becomes easier while actually resulting in a stronger statement that avoids the need for the “buffer”. We therefore highly recommend that anyone interested in further exploring this topic use that setup instead of the one employed so far in this and the previous papers. Note that the cyclic setting is for example also used in~\cite{NSV18}, where a supersaturation-type Hales--Jewett statement is obtained for combinatorial lines whose set of active coordinates is contained in some small interval.

\bigskip

\noindent {\bf Acknowledgements. }We would like to thank David Conlon and Oriol Serra for helpful discussions and comments on the topic. The first author would like to thank UPC for their hospitality during her visit to Barcelona. Part of the research was conducted while the first author was at the Department of Mathematics of ETH Z\"urich.

\bibliography{bib}

\begin{thebibliography}{10}

\bibitem{Be68}
E.~R. Berlekamp.
\newblock A construction for partitions which avoid long arithmetic
  progressions.
\newblock {\em Canad. Math. Bull}, 11(1968):409--414, 1968.

\bibitem{BCT18}
T.~Blankenship, J.~Cummings, and V.~Taranchuk.
\newblock A new lower bound for van der waerden numbers.
\newblock {\em European Journal of Combinatorics}, 69:163--168, 2018.

\bibitem{Co18}
D.~Conlon.
\newblock Monochromatic combinatorial lines of length three.
\newblock {\em arXiv preprint arXiv:1810.09767}, 2018.

\bibitem{CK18}
D.~Conlon and N.~Kam\v{c}ev.
\newblock Intervals in the {H}ales--{J}ewett theorem.
\newblock {\em arXiv preprint arXiv:1801.08919}, 2018.

\bibitem{GRS90}
R.~L. Graham, R.~L. Graham, B.~L. Rothschild, and J.~H. Spencer.
\newblock {\em Ramsey theory}, volume~20.
\newblock John Wiley \& Sons, 1990.

\bibitem{HJ63}
A.~Hales and R.~Jewett.
\newblock Regularity and positional games.
\newblock {\em Transactions of the American Mathematical Society},
  106(2):222--229, 1963.

\bibitem{HT14}
N.~Hindman and E.~Tressler.
\newblock The first nontrivial {H}ales--{J}ewett number is four.
\newblock {\em Ars Comb.}, 113:385--390, 2014.

\bibitem{La16}
M.~Lavrov.
\newblock An upper bound for the {H}ales--{J}ewett number {HJ}(4,2).
\newblock {\em SIAM Journal on Discrete Mathematics}, 30(2):1333--1342, 2016.

\bibitem{LR18}
I.~Leader and E.~R{\"a}ty.
\newblock A note on intervals in the {H}ales--{J}ewett {T}heorem.
\newblock {\em arXiv preprint arXiv:1802.03087}, 2018.

\bibitem{NSV18}
J.~Ne\v{s}et\v{r}il, C.~Spiegel, and L.~Vena.
\newblock On the asymptotic behaviour of the number of monochromatic
  configurations.
\newblock {\em In preparation}.

\bibitem{Sh88}
S.~Shelah.
\newblock Primitive recursive bounds for van der {W}aerden numbers.
\newblock {\em Journal of the American Mathematical Society}, 1(3):683--697,
  1988.

\end{thebibliography}
\bibliographystyle{abbrv}

\end{document}